\documentclass[reqno,12pt]{amsart}
\usepackage{amsfonts}
\usepackage{bbm}
\usepackage{}
\numberwithin{equation}{section}
\usepackage{hyperref}
\usepackage{color}
\usepackage{rotating}
\usepackage{indentfirst}
\usepackage{color}
\usepackage{amssymb}
\usepackage{mathrsfs}
\usepackage{xy}
\usepackage{rotating}
\usepackage{tikz}

\xyoption{all}

\theoremstyle{plain}
\newtheorem{theorem}{\bf Theorem}[section]
\newtheorem{lemma}[theorem]{\bf Lemma}
\newtheorem{corollary}[theorem]{\bf Corollary}
\newtheorem{proposition}[theorem]{\bf Proposition}

\theoremstyle{definition}
\newtheorem{definition}[theorem]{\bf Definition}
\newtheorem{remark}[theorem]{\bf Remark}
\newtheorem{example}[theorem]{\bf Example}

\newcommand{\bt}{\begin{theorem}}
\newcommand{\et}{\end{theorem}}
\newcommand{\bl}{\begin{lemma}}
\newcommand{\el}{\end{lemma}}
\newcommand{\bd}{\begin{definition}}
\newcommand{\ed}{\end{definition}}
\newcommand{\bc}{\begin{corollary}}
\newcommand{\ec}{\end{corollary}}
\newcommand{\bp}{\begin{proof}}
\newcommand{\ep}{\end{proof}}
\newcommand{\bx}{\begin{example}}
\newcommand{\ex}{\end{example}}
\newcommand{\br}{\begin{remark}}
\newcommand{\er}{\end{remark}}
\newcommand{\be}{\begin{equation}}
\newcommand{\ee}{\end{equation}}
\newcommand{\ba}{\begin{align}}
\newcommand{\ea}{\end{align}}
\newcommand{\bn}{\begin{enumerate}}
\newcommand{\en}{\end{enumerate}}
\newcommand{\bcs}{\begin{cases}}
\newcommand{\ecs}{\end{cases}}
\newcommand{\RNum}[1]{\uppercase\expandafter{\romannumeral #1\relax}}

\makeatletter
\renewcommand{\section}{\@startsection{section}{1}{0mm}
  {-\baselineskip}{0.5\baselineskip}{\bf\leftline}}
\makeatother

\begin{document}
\sloppy
\title[ENDOMORPHISM ALGEBRAS OF SILTING COMPLEXES]{Endomorphism algebras of  silting complexes}
\author{Lidia Angeleri H\"{u}gel}
\address{Universit\'a degli Studi di Verona, Strada Le Grazie 15, 37134 Verona, Italia} \par
 \email{lidia.angeleri@univr.it} 
\author{Marcelo Lanzilotta}
  \address{Departamento de Matem\'atica, Facultad de Ingenier\'ia, Universidad de la Rep\'ublica, Uruguay} 
  \email{marclan@fing.edu.uy}
  \author[J. Liu]{Jifen Liu$^{\ast}$}
\address{Institute of Mathematics, School of Mathematical Sciences, Nanjing Normal University, Nanjing 210023, P. R. China.}
\email{liujifen24@163.com}
\author{Sonia Trepode}
  \address{Centro Marplatense de Investigaciones Matem\'aticas. FCEyN, Universidad Nacional de Mar del Plata, CONICET. Dean Funes 3350, Mar del Plata, Argentina.} \par
  \email{strepode@mdp.edu.ar}

\thanks{$\ast$: Corresponding author.}

\begin{abstract}
	We consider endomorphism algebras of $n$-term silting complexes in derived categories of hereditary  algebras, and we show that the module category of such an endomorphism algebra has a separated $n$-section. For $n=3$ we obtain a trisection in the sense of \cite{AACPT}.
\end{abstract}

\maketitle
\section{Introduction}

In representation theory, algebras are studied in terms of their module category.  The category of finite dimensional modules over a  finite dimensional
hereditary algebra is rather well understood, and it  is often taken as a starting point for exploring more complex situations. For example, Happel and Ringel \cite{HR} studied
tilted algebras, the endomorphism algebras of tilting modules over hereditary finite dimensional algebras. Their module categories can be completely described as a ``tilt'' of the module category of the underlying hereditary algebra. Later, Happel, Reiten and Smal\o\ \cite{HRS} extended these results to the class of quasi-tilted algebras, the algebras occurring as endomorphism algebras of tilting objects in hereditary abelian categories with finiteness conditions. 
They characterized quasi-tilted algebras homologically as the algebras of global dimension at most two such that each indecomposable module has projective or injective dimension at most one.
 This led Coelho and Lanzilotta \cite{CL} to investigate shod algebras, which are defined as the algebras satisfying the latter homological condition on indecomposable modules. Shod algebras always have global dimension at most three, the ones of global dimension three are called strictly shod. 

In 2016, Buan and Zhou \cite{BZ} showed that shod algebras admit a very natural characterization in terms of  the notion of a silting complex introduced in \cite{KV}. They proved that the strictly shod algebras are precisely the silted algebras, that is, the endomorphism algebras of 2-term silting complexes in the bounded derived category of a hereditary finite dimensional algebra.

The purpose of this work is to investigate $n$-{\it silted} algebras, i.e.~endomorphism algebras of $n$-term silting complexes in derived categories of hereditary  algebras. 
We will prove that the module category of an $n$-silted algebra has a  separated $n$-section.
To this end, we will employ the fact that the module category of the endomorphism ring of a silting complex is equivalent to the heart of an associated t-structure. We will thus work in the heart of the t-structure induced by our $n$-term silting complex. 
A crucial role will be played by the connection developed in \cite{AH} between silting complexes and chains of homological ring epimorphisms over hereditary algebras. 

The case $n=3$ is particularly nice. We start  with a 3-term silting complex in the derived category of a hereditary algebra, and we prove that its endomorphism algebra has a separated trisection in its module category given by three functorially finite subcategories. Finally, we also prove that every functorially finite $n$-section over a hereditary algebra, under mild conditions, is associated to an $n$-term silting complex.

\section{Preliminaries}

\subsection{Notation}
Throughout this paper, let $A$ be a finite dimensional algebra over a field $k$. We always assume that all modules are right modules. A composition $g\circ f$ of morphisms $f$ and $g$ means first $f$ then $g$. But a composition $\alpha\beta$ of arrows $\alpha$ and $\beta$ means that first $\alpha$ then $\beta$. The category of all right $A$-modules is denoted by $\mathrm{Mod}A$, and the subcategory of finitely presented $A$-modules is denoted by $\mathrm{mod}A$. For a given $A$-module $X$, we denote by $\mathrm{pd}X$ (resp. $\mathrm{id}X$) the projective (resp. injective) dimension of $X$.

Let $\mathcal{M}\subset \mathrm{Mod}A$ be a class of modules. $\mathrm{Add}\mathcal{M}$ (resp. $\mathrm{add}\mathcal{M}$) denotes the class consisting of all modules isomorphic to direct summands of (finite) direct sums of elements in $\mathcal{M}$, while $\mathrm{Gen}\mathcal{M}$ (resp. $\mathrm{gen}\mathcal{M}$) is the class of epimorphic images of (finite) direct sums of elements in $\mathcal{M}$. Dually, we define $\mathrm{Cogen}\mathcal{M}$ (resp. $\mathrm{cogen}\mathcal{M}$) as the class of all submodules of (finite) direct sums of elements in $\mathcal{M}$.

Denote by $\mathrm{ind}A$ the subcategory of $\mathrm{mod}A$ formed by the indecomposable $A$-modules. Given $X,Y\in \mathrm{ind}A$, a path from $X$ to $Y$ in $\mathrm{ind}A$ is a sequence of non-zero morphisms $X=X_{0}\rightarrow X_{1}\rightarrow \cdots \rightarrow X_{t-1}\rightarrow X_{t}=Y \ (t\geq 1)$, where $X_{i}\in \mathrm{ind}A$ for all $i$. We say that $X$ is a predecessor of $Y$ and $Y$ is a successor of $X$.

Let $\mathcal{C}$ be a subcategory of $\mathrm{ind}A$. Recall that $\mathcal{C}$ is closed under predecessors if, whenever there is a path from $X$ to $Y$ in $\mathrm{ind}A$, with $Y\in \mathcal{C}$, then $X\in \mathcal{C}$. An example is the left part $\mathcal{L}_{A}$ of $\mathrm{mod}A$, defined in \cite{HRS}, which is a full subcategory of $\mathrm{ind}A$ with object class 
$$\mathcal{L}_{A}=\left\{Y\in \mathrm{ind}A:\mathrm{pd}X\leq 1 \ \text{whenever there is a path from} \ X \ \text{to} \ Y \right\}.$$
Dually, we define subcategories closed under successors. An example of such categories is the right part $\mathcal{R}_{A}$ of $\mathrm{ind}A$
$$\mathcal{R}_{A}=\left\{X\in \mathrm{ind}A:\mathrm{id}Y\leq 1 \ \text{whenever there is a path from} \ X \ \text{to} \ Y \right\}.$$

\subsection{Functorially finite subcategories} Let $\mathcal{X}$ be a subcategory of an additive category $\mathcal A$.  For an object $M$ in $\mathcal A$, a {\it right $\mathcal{X}$-approximation} of $M$ is a morphism $f:X\longrightarrow M$ with $X\in \mathcal{X}$ such that any morphism $f^{\prime}:X^{\prime}\longrightarrow M$ with $X^{\prime}\in \mathcal{X}$ factors through $f$. If every object in $\mathcal A$ has a right $\mathcal{X}$-approximation, we call $\mathcal{X}$ {\it contravariantly finite} in $\mathcal A$. The notions of a {\it left $\mathcal{X}$-approximation} and a {\it covariantly finite subcategory} are defined dually. We say $\mathcal{X}$ is {\it functorially finite} if it is both contravariantly finite and covariantly finite.

\begin{proposition}\cite{AS}
	Let $(\mathcal{T},\mathcal{F})$ be a torsion pair in $\mathrm{mod}A$. The following are equivalent:

	$(1)$ The torsion class $\mathcal{T}$ is functorially finite;

	$(2)$ There exists $M\in \mathrm{mod}A$ such that $\mathcal{T}=\mathrm{gen}M$;

	$(3)$ The torsion-free class $\mathcal{F}$ is functorially finite;

	$(4)$ There exists $N\in \mathrm{mod}A$ such that $\mathcal{F}=\mathrm{cogen}N$.
\end{proposition}

A torsion pair fulfilling the equivalent properties in the above proposition is a {\it functorially finite} torsion pair.

\subsection{Suspended subcategories and t-structures} Let $\mathcal{T}$ be a triangulated category with shift functor $[1]$. For a full subcategory $\mathcal{V}$ of $\mathcal{T}$ and a subset $I\subseteq \mathbb{Z}$ (which is usually expressed by symbols such as $\geq n, \leq n$, or just $n$), we set
$$\mathcal{V}^{\bot_{I}}=\left\{W\in \mathcal{T} \ \vert \ \mathrm{Hom}_{\mathcal{T}}(V,W[i])=0 \ \text{for any} \ i\in I \ \text{and} \ V\in \mathcal{V} \right\},$$
$$^{\bot_{I}}\mathcal{V}=\left\{W\in \mathcal{T} \ \vert \ \mathrm{Hom}_{\mathcal{T}}(W,V[i])=0 \ \text{for any} \ i\in I \ \text{and} \ V\in \mathcal{V} \right\}.$$
For instance,
$$\mathcal{V}^{\bot_{0}}=\left\{W\in \mathcal{T} \ \vert \ \mathrm{Hom}_{\mathcal{T}}(V,W)=0 \ \text{for any} \ V\in \mathcal{V} \right\},$$
$$^{\bot_{0}}\mathcal{V}=\left\{W\in \mathcal{T} \ \vert \ \mathrm{Hom}_{\mathcal{T}}(W,V)=0 \ \text{for any} \ V\in \mathcal{V} \right\}.$$

A full subcategory $\mathcal{V}$ of $\mathcal{T}$, closed under direct summands, is said to be {\it suspended} if it is closed under positive shifts and extensions, that is, if $X\in \mathcal{V}$, then $X[i]\in \mathcal{V}$ for all integers $i>0$, and if $X\rightarrow Y\rightarrow Z\rightarrow X[1]$ is a triangle in $\mathcal{T}$ with $X,Z\in \mathcal{V}$ then $Y\in \mathcal{V}$. Dually, one can define {\it cosuspended} subcategories.

\begin{definition}\cite{BBD}
	A {\it t-structure} on $\mathcal{T}$ is a pair of full subcategories $(\mathcal{V},\mathcal{W})$ closed under direct summands such that

	$(1)$ $\mathrm{Hom}_{\mathcal{T}}(\mathcal{V},\mathcal{W})=0$, i.e. $\mathrm{Hom}_{\mathcal{T}}(V,W)=0$ for any $V\in \mathcal{V}$ and $W\in \mathcal{W}$;

	$(2)$ $\mathcal{V}[1]\subseteq \mathcal{V}$;

	$(3)$ for any $X$ in $\mathcal{T}$, there exist $V\in \mathcal{V}$, $W\in \mathcal{W}$ and a triangle $V\rightarrow X\rightarrow W\rightarrow V[1]$.
\end{definition}

A suspended subcategory $\mathcal{V}$ of $\mathcal{T}$ is called an {\it aisle} if the inclusion functor $\mathcal{V}\hookrightarrow \mathcal{T}$ has a right adjoint $u:\mathcal{T}\rightarrow \mathcal{V}$. Similarly, a cosuspended subcategory $\mathcal{W}$ of $\mathcal{T}$ is called a {\it coaisle} if the inclusion functor $\mathcal{W}\hookrightarrow \mathcal{T}$ has a left adjoint $v:\mathcal{T}\rightarrow \mathcal{W}$. It is shown in \cite{KV} that the following conditions are equivalent for a suspended subcategory $\mathcal{V}$ of $\mathcal{T}$:

$(i)$ $\mathcal{V}$ is an aisle.

$(ii)$ $(\mathcal{V},\mathcal{V}^{\bot_{0}})$ is a t-structure.

$(iii)$ For any $X$ in $\mathcal{T}$, there is a triangle $V\rightarrow X\rightarrow W\rightarrow V[1]$ with $V\in \mathcal{V}$ and $W\in \mathcal{V}^{\bot_{0}}$.

	$(iv)$ $\mathcal{V}$ is contravariantly finite in $\mathcal{T}$. 

Recall that the {\it heart} of the t-structure $(\mathcal{V},\mathcal{W})$ is the subcategory $\mathcal{H}=\mathcal{V}\cap\mathcal{W}[1]$. It is an abelian category by \cite{BBD}. We denote by $H^{0}:\mathcal{T}\rightarrow \mathcal{H}$ the associated cohomological functor given by $H^{0}(X)=u(v(X)[1])$.

\vskip 10pt

The following Lemma will be useful later.

\begin{lemma}\cite{GM}\label{LM}
	Let $(\mathcal{V},\mathcal{W})$ be a t-structure in $\mathcal{T}$ with heart $\mathcal{H}=\mathcal{V}\cap\mathcal{W}[1]$. Given a morphism $f:X\rightarrow Y$ in $\mathcal{H}$, let $Z$ be the cone of $f$ in $\mathcal{T}$. Consider the canonical triangle
	$$K\rightarrow Z\rightarrow L\rightarrow K[1]$$
	with $K\in \mathcal{V}[1]$ and $L\in \mathcal{W}[1]$. Then $\mathrm{Ker}_{\mathcal{H}}f=K[-1]$ and $\mathrm{Coker}_{\mathcal{H}}f=L$.
\end{lemma}

\subsection{Silting complexes} Silting complexes were introduced in \cite{KV} to study t-structures in the derived category of a hereditary algebra.

\begin{definition}\cite{PV}
	Let $D(A)=D(\mathrm{Mod}A)$ be the unbounded derived category of $\mathrm{Mod}A$. An object $T$ in $D(A)$ is {\it silting} if the pair $(T^{\bot_{>0}},T^{\bot_{\leq 0}})$ is a t-structure in $D(A)$, which we call the {\it silting t-structure} induced by $T$.

	Two silting objects $T, T^{\prime}$ in $D(A)$ are {\it equivalent} if they induce the same t-structure.
	 
\end{definition}

It is shown in \cite[Proposition 4.2]{AMV} that a bounded complex of finitely generated projective $A$-modules $T$ in $K^{b}(\mathrm{proj}A)$ is a {\it silting object} if and only if it satisfies the following conditions:

$(S1)$ $\mathrm{Hom}_{D(A)}(T,T[i])=0$ for all $i>0$;

$(S2)$ $T$ is a generator of $D(A)$, i.e. $T^{\bot_{\mathbb{Z}}}=0$.

We call $T$ a {\it bounded silting complex}.

An $n$-term silting complex is a bounded silting complex with $n$ non-zero terms, which we always assume to be concentrated in degrees $0,\dots,-n+1$. If $T$ is a $2$-term silting complex in $K^{b}(\mathrm{proj}A)$, then its cohomology in degree zero $H^{0}(T)$ is called a {\it silting module} \cite{AMV}. Over a finite dimensional algebra, the finite dimensional silting modules are precisely the support $\tau$-tilting modules from \cite{AIR}. Every silting module generates a torsion class, called {\it silting class}.

\subsection{Ring epimorphisms}  

\begin{definition}
	A ring homomorphism $\lambda:A\rightarrow B$ is a {\it ring epimorphism} if it is an epimorphism in the category of rings with unit, or equivalently, if the functor given by restriction of scalars $\lambda_{\ast}:\mathrm{Mod}B\rightarrow \mathrm{Mod}A$ is a full embedding.

	A ring epimorphism $\lambda:A\rightarrow B$ is said to be

	$(i)$ {\it homological} if $\mathrm{Tor}^{A}_{i}(B,B)=0$ for all $i>0$, or equivalently, the functor given by restriction of scalars $\lambda_{\ast}:D(B)\rightarrow D(A)$ is a full embedding.

	$(ii)$ {\it pseudoflat} if $\mathrm{Tor}^{A}_{1}(B,B)=0$.

	Two ring epimorphisms $\lambda_{1}:A\rightarrow B_{1}$ and $\lambda_{2}:A\rightarrow B_{2}$ are {\it equivalent} if there is an isomorphism of rings $\mu:B_{1}\rightarrow B_{2}$ such that $\lambda_{2}=\mu\circ \lambda_{1}$. We say that $\lambda_{1}$ and $\lambda_{2}$ lie in the same {\it epiclass} of $A$. 
	
	Epiclasses of a ring $A$ can be classified by suitable subcategories of $\mathrm{Mod}A$. 
\end{definition}

\begin{definition}
	A full subcategory $\mathcal{X}$ of $\mathrm{Mod}A$ is called {\it bireflective} if the inclusion functor $\mathcal{X}\rightarrow \mathrm{Mod}A$ admits both a left and a right adjoint, or equivalently, $\mathcal{X}$ is closed under products, coproducts, kernels and cokernels.
\end{definition}

\begin{theorem}\cite{Gd,BD}\label{Th1}
	The assignment which takes a ring epimorphism $\lambda:A\rightarrow B$ to the essential image $\mathcal{X}_{B}$ of $\lambda_{\ast}$ defines a bijection between:

	$(1)$ epiclasses of ring epimorphisms $A\rightarrow B$,

	$(2)$ bireflective subcategories of $\mathrm{Mod}A$,

	which restricts to a bijection between

	$(1)^{\prime}$ epiclasses of pseudoflat ring epimorphisms $A\rightarrow B$,

	$(2)^{\prime}$ bireflective subcategories closed under extensions in $\mathrm{Mod}A$.
\end{theorem}

In particular, if $A$ is a hereditary ring, then $\lambda:A\rightarrow B$ is a homological ring epimorphism if and only if it is pseudoflat, which is equivalent to being a universal localization of $A$ by \cite[Theorem 6.1]{K}. This shows that universal localization provides a powerful tool to construct homological ring epimorphisms for hereditary rings.

\begin{theorem}\cite[Theorem 4.1]{S}
	Let $A$ be a ring and $\Sigma$ be a class of morphisms between finitely generated projective right $A$-modules. Then there is a pseudoflat ring epimorphism $\lambda:A\rightarrow A_{\Sigma}$ called the universal localization of $A$ at $\Sigma$ such that

	$(1)$ $\lambda$ is $\Sigma$-inverting, i.e. if $\sigma$ belongs to $\Sigma$, then $\sigma\otimes_{A}A_{\Sigma}$ is an isomorphism of right $A_{\Sigma}$-modules, and

	$(2)$ $\lambda$ is universal $\Sigma$-inverting, i.e. for any $\Sigma$-inverting morphism $\lambda^{\prime}:A\rightarrow B$ there exists a unique ring homomorphism $g:A_{\Sigma}\rightarrow B$ such that $g\circ \lambda=\lambda^{\prime}$.
\end{theorem}

Moreover, every pseudoflat ring epimorphism $\lambda:A\rightarrow B$ starting in a hereditary ring $A$ induces a silting module $T=B\oplus \mathrm{Coker}\lambda$, see for example \cite{AMV2}.

\section{$n$-term silting complexes over hereditary algebras }

From now on, we assume that $A$ is a hereditary algebra. We want to exploit a result from \cite{AH} stating that bounded silting complexes are closely related to ring epimorphisms.

\subsection{The t-structure induced by a silting complex}

The partial order on bireflective subcategories given by inclusion corresponds, under the bijection in Theorem \ref{Th1}, to a partial order on the epiclasses of $A$ defined by setting $\lambda_{1}\leq \lambda_{2}$ whenever $\lambda_{1}:A\longrightarrow B_{1}$ factors through $\lambda_{2}:A\longrightarrow B_{2}$ via a ring homomorphism $\mu:B_{2}\longrightarrow B_{1}$, that is, $\lambda_{1}=\mu\circ \lambda_{2}$.

\vskip 10pt

It is shown in \cite[Section 5]{AH} that every chain
$$\cdots\leq\lambda_{n-1}\leq\lambda_{n}\leq\lambda_{n+1}\leq \cdots$$
of homological ring epimorphisms $\lambda_{n}:A\longrightarrow B_{n}$ induces a t-structure in $D(A)$. More precisely, consider the corresponding bireflective subcategories $\mathcal{X}_{n}$ of Mod$A$, which are all extension closed, together with the silting classes $$\mathcal{D}_{n}=\mathrm{Gen}(\mathcal{X}_{n})$$ induced by the silting $A$-modules $T_{n}=B_{n}\bigoplus \mathrm{Coker}(\lambda_{n})$. We set $$\mathcal{V}_{n}=\mathcal{D}_{n}\cap \mathcal{X}_{n+1}$$ for all $n\in \mathbb{Z}$. Then there is a t-structure $(\mathcal{V},\mathcal{W})$ in $D(A)$ with aisle
$$\mathcal{V}=\left\{X\in D(A) \ \vert \ H^{-n}(X)\in \mathcal{V}_{n} \ \text{for all} \ n\in \mathbb{Z}\right\}.$$

The following proposition shows that the t-structure is induced by a silting complex under suitable hypotheses.

\begin{proposition}\cite[Proposition 5.15]{AH}\label{Pro}
	Let $A$ be a hereditary algebra and let $\cdots\leq\lambda_{n-1}\leq\lambda_{n}\leq\lambda_{n+1}\leq \cdots$ be a chain of homological ring epimorphisms $\lambda_{n}:A\longrightarrow B_{n}$ with induced ring epimorphisms $\mu_{n}:B_{n+1}\longrightarrow B_{n}$ given by the commutative diagram

	$$\xymatrix{A\ar[rd]_{\lambda_{n+1}} & &B_{n}\\
	 &B_{n+1}\ar[ur]_{\mu_{n}}& 
	 \ar"1,1";"1,3"^{\lambda_{n}}}$$

	 \noindent Let $\mathcal{X}_{n}$ be the corresponding extension closed bireflective subcategories of Mod$A$ and let $\mathcal{K}_{n}=\mathrm{Ker}\mathbf{R}Hom_{A}(B_{n},-)$ be the triangulated subcategories of $D(A)$ associated with $\lambda_{n}$. Then the  t-structure $(\mathcal{V},\mathcal{W})$ in $D(A)$ is induced by a silting object if and only if the conditions
	 $$\bigcap_{n\in \mathbb{Z}}\mathcal{X}_{n}=0 \ and \ \bigcap_{n\in \mathbb{Z}}\mathcal{K}_{n}=0$$
	 hold true. In this case, the t-structure $(\mathcal{V},\mathcal{W})$ is induced by the silting object 
	 $$T=\bigoplus_{n\in \mathbb{Z}}\mathrm{Cone}(\mu_{n})[n].$$

\end{proposition}

Every bounded silting complex over a hereditary algebra $A$ has this form, as explained in the following theorem.

\begin{theorem} \cite[Theorem 6.7]{AH}\label{TH}
	Let $A$ be a finite dimensional hereditary algebra. Every bounded silting complex $T$ in $K^{b}(\mathrm{proj}A)$ arises as in Proposition \ref{Pro} from a finite chain of finite dimensional homological ring epimorphisms $0_{A}\leq\lambda_{n}\leq \cdots \leq\lambda_{m}\leq id_{A}$.
\end{theorem}

\subsection{The heart of the t-structure}
Let now $T$ be an $n$-term silting complex in $K^{b}(\mathrm{proj}A)$. According to Theorem \ref{TH}, $T$ arises from a chain of finite dimensional homological ring epimorphisms
$$0_{A}\leq \lambda_{0}\leq \lambda_{1}\leq \cdots \leq \lambda_{n-2}\leq id_{A}$$
with $\lambda_{i}:A\longrightarrow B_{i}$ for integers $0\leq i\leq n-2$. Consider the corresponding chains of extension-closed bireflective subcategories
$$0\subseteq \mathcal{X}_{0}\subseteq \mathcal{X}_{1}\subseteq \cdots \subseteq\mathcal{X}_{n-2}\subseteq \mathrm{Mod}A$$
and of silting classes
$$0\subseteq \mathcal{D}_{0}\subseteq \mathcal{D}_{1}\subseteq \cdots \subseteq\mathcal{D}_{n-2}\subseteq \mathrm{Mod}A.$$
We construct a chain
$$0\subseteq \mathcal{V}_{0}\subseteq \mathcal{V}_{1}\subseteq \cdots \subseteq\mathcal{V}_{n-2}\subseteq \mathrm{Mod}A$$
with 
$$\mathcal{V}_{i}=\mathcal{D}_{i}\cap\mathcal{X}_{i+1}=\mathrm{Gen}(\mathcal{X}_{i})\cap \mathcal{X}_{i+1}=\mathrm{Gen}B_{i}\cap \mathrm{Mod}B_{i+1} \ \text{for} \ i=0,1,\dots,n-3$$
and
$$\mathcal{V}_{n-2}=\mathcal{D}_{n-2}\cap \mathcal{X}_{n-1}=\mathcal{D}_{n-2}\cap \mathrm{Mod}A=\mathrm{Gen}(\mathcal{X}_{n-2})=\mathrm{Gen}B_{n-2}.$$
These classes induce a t-structure $(\mathcal{V},\mathcal{W})$ in $D(A)$, where the aisle
$$\mathcal{V}=\left\{X\in D(A) \ \vert \ H^{-j}(X)\in \mathcal{V}_{j} \ \text{for all} \ j\in \mathbb{Z}\right\}$$
consists of the complexes $X$ with cohomologies concentrated in degrees $\leq 0$ satisfying $H^{-i}(X)\in \mathrm{Gen}B_{i}\cap \mathrm{Mod}B_{i+1} \ \text{for} \ i=0,1,\dots,n-3$ and $H^{-(n-2)}(X)\in \mathrm{Gen}B_{n-2}$.

Now we compute the coaisle $\mathcal{W}$. Since $A$ is hereditary, we know from \cite[Section 3.3]{AH} that $\mathcal{W}$ is determined by its cohomologies, that is,
$$\mathcal{W}=\left\{X\in D(A) \ \vert \ H^{j}(X)\in \mathcal{W}_{j} \ \text{for all} \ j\in \mathbb{Z}\right\}$$
where $\mathcal{W}_{j}=H^{j}(\mathcal{W})$ satisfies $\mathcal{W}_{j}={\mathcal{V}_{-j}}^{\bot_{0}}\cap {\mathcal{V}_{-(j+1)}}^{\bot_{1}}$.

\noindent Hence we have a chain 
$$0\subseteq \mathcal{W}_{-(n-2)}\subseteq\cdots \subseteq \mathcal{W}_{-1}\subseteq \mathcal{W}_{0}\subseteq \mathrm{Mod}A$$
with 
$$\mathcal{W}_{-j}={\mathcal{V}_{j}}^{\bot_{0}}\cap {\mathcal{V}_{j-1}}^{\bot_{1}} \ \text{for} \ j=1,2,\dots,n-2$$
and
$$\mathcal{W}_{0}={\mathcal{V}_{0}}^{\bot_{0}}\cap {\mathcal{V}_{-1}}^{\bot_{1}}={\mathcal{V}_{0}}^{\bot_{0}}\cap \mathrm{Mod}A={\mathcal{V}_{0}}^{\bot_{0}}.$$

We compute the heart of the t-structure $(\mathcal{V},\mathcal{W})$. Since $\mathcal{W}[1]$ consists of the complexes $X$ with $H^{j}(X)\in \mathcal{W}_{j+1}$ for all $j\in \mathbb{Z}$, we have
$$\mathcal{H}=\mathcal{V}\cap \mathcal{W}[1]=\left\{X\in D(A) \ \vert \ H^{j}(X)\in \mathcal{V}_{-j}\cap\mathcal{W}_{j+1} \ \text{for all} \ j\in \mathbb{Z} \right\}.$$
Hence a complex $X$ lies in $\mathcal{H}$ if and only if 
$$H^{0}(X)\in \mathcal{V}_{0}\cap\mathcal{W}_{1}=\mathcal{V}_{0}\cap \mathrm{Mod}A=\mathcal{V}_{0},$$
$$H^{-j}(X)\in \mathcal{V}_{j}\cap\mathcal{W}_{-(j-1)}\ \text{for} \ j=1,2,\dots,n-2,$$
$$H^{-(n-1)}(X)\in \mathcal{V}_{n-1}\cap \mathcal{W}_{-(n-2)}=\mathrm{Mod}A \cap \mathcal{W}_{-(n-2)}=\mathcal{W}_{-(n-2)},$$
$$H^{j}(X)=0 \ \text{for all} \ j\in \mathbb{Z}\setminus \left\{{0,-1,-2, \dots, -(n-1)}\right\}.$$

The objects in $\mathcal{H}$, being directs sums of their cohomologies, decompose as direct sums of stalk complexes in $\mathcal{V}_{0}[0], (\mathcal{V}_{j}\cap\mathcal{W}_{-(j-1)})[j]$ for $j=1,2,\dots,n-2$ and $\mathcal{W}_{-(n-2)}[n-1]$. We will denote these classes by $\mathcal{H}_{j}$, $0\le j<n$. We have the following result.

\begin{proposition}
	Let $A$ be a finite dimensional hereditary algebra, and let $T$ be an $n$-term silting complex in $K^{b}(\mathrm{proj}A)$. Then the heart of the t-structure $(\mathcal{V},\mathcal{W})$ in $D(A)$ induced by $T$ has a decomposition $\mathcal{H}=\mathcal{H}_{0}\vee \mathcal{H}_{1}\vee\dots\vee \mathcal{H}_{n-2}\vee \mathcal{H}_{n-1}$.
\end{proposition}

\subsection{The $n$-section of the heart}

We introduce the definition of an $n$-section in $\mathrm{ind}A$.

\begin{definition}
	Let $A$ be an artin algebra. An {\it $n$-section} $(\mathcal{A}_{1},\cdots,\mathcal{A}_{n})$ of $\mathrm{ind}A$ is an $n$-tuple of non-empty disjoint full subcategories of $\mathrm{ind}A$ such that

	$(1)$ $\mathrm{ind}A=\mathcal{A}_{1}\cup\cdots\cup \mathcal{A}_{n}$;

	$(2)$ $\mathrm{Hom}_{A}(\mathcal{A}_{i},\mathcal{A}_{j})=0$ for $1\leq j<i\leq n$.
	
\end{definition}
In particular, for $n=2$ we obtain a split torsion pair, and for $n=3$ a  trisection in the sense of \cite{AACPT,HRS}.

\vskip 10pt

Motivated by \cite[p.120]{R}, we say that an $n$-section $(\mathcal{A}_{1},\cdots,\mathcal{A}_{n})$ is {\it separated} if, for any three adjacent subcategories $(\mathcal{A}_{i-1},\mathcal{A}_{i},\mathcal{A}_{i+1})\  (2\leq i\leq n-1)$ of $\mathrm{ind}A$, any morphism $A_{i-1}\longrightarrow A_{i+1}$ with $A_{i-1}\in \mathcal{A}_{i-1}$ and $A_{i+1}\in \mathcal{A}_{i+1}$ factors through $\mathrm{add}\mathcal{A}_{i}$. Moreover, an $n$-section $(\mathcal{A}_{1},\cdots,\mathcal{A}_{n})$ is called {\it functorially finite} if every $\mathrm{add}\mathcal{A}_{i}$ is a functorially finite subcategory of $\mathrm{mod}A$.

\vskip 10pt

Let $A$ be a finite dimensional hereditary algebra, and let $T$ be an $n$-term silting complex in $K^{b}(\mathrm{proj}A)$ with endomorphism ring $B=\mathrm{End}_{D(A)}T$. Following \cite{KD,PV}, the heart of the t-structure $(\mathcal{V},\mathcal{W})$ induced by $T$ is equivalent to the module category $\mathrm{Mod}B$ via the functor $F=\mathrm{Hom}_{D(A)}(T,-)\arrowvert_{\mathcal{H}}:\mathcal{H}\longrightarrow \mathrm{Mod}B$. Moreover, the t-structure $(\mathcal{V},\mathcal{W})$ restricts to a bounded t-structure $(\mathcal{V}^{\prime},\mathcal{W}^{\prime})$ in $D^{b}(\mathrm{mod}A)$ and $F$ restricts to an equivalence between the heart $\mathcal{H}^{\prime}=\mathcal{H}\cap D^{b}(\mathrm{mod}A)$ and $\mathrm{mod}B$. Put 
$$\mathcal{B}_{0}=F((\mathcal{V}_{0}\cap \mathrm{mod}A)[0]), \ \mathcal{B}_{1}=F((\mathcal{V}_{1}\cap\mathcal{W}_{0}\cap \mathrm{mod}A)[1]), \dots,$$ $$\mathcal{B}_{n-2}=F((\mathcal{V}_{n-2}\cap\mathcal{W}_{-(n-3)}\cap \mathrm{mod}A)[n-2]),  \mathcal{B}_{n-1}=F((\mathcal{W}_{-(n-2)}\cap \mathrm{mod}A)[n-1]).$$
Then $\mathcal{H}^{\prime}$ is a Krull-Schmidt $k$-category, even a length category, and
$$\mathrm{ind}B=\mathrm{ind}\mathcal{B}_{0}\cup \mathrm{ind}\mathcal{B}_{1}\cup \cdots \cup \mathrm{ind}\mathcal{B}_{n-1}$$

\begin{proposition}\label{P1}
	$(1)$ $(\mathrm{ind}\mathcal{B}_{0}, \mathrm{ind}\mathcal{B}_{1}, \dots, \mathrm{ind}\mathcal{B}_{n-1})$ is a separated $n$-section of $\mathrm{ind}B$ with $\mathrm{ind}\mathcal{B}_{0}$ closed under predecessors and $\mathrm{ind}\mathcal{B}_{n-1}$ closed under successors.

	$(2)$ $\mathrm{id}X_{B}\leq 1$ for any object $X$ in $\mathcal{B}_{n-1}$.

	$(3)$ $\mathrm{pd}X_{B}\leq 1$ for any object $X$ in $\mathcal{B}_{0}$. 
\end{proposition}

\begin{proof}
	$(1)$ We have seen that $\mathrm{ind}B=\mathrm{ind}\mathcal{B}_{0}\cup \mathrm{ind}\mathcal{B}_{1}\cup \cdots \cup \mathrm{ind}\mathcal{B}_{n-1}$, and by construction, there are only morphisms from the left to the right. It is obvious that $\mathrm{ind}\mathcal{B}_{0}$ is closed under predecessors and $\mathrm{ind}\mathcal{B}_{n-1}$ is closed under successors. On the other hand, $\mathrm{Hom}_{D(A)}(X,Y[2])=0$ for any $X,Y\in \mathrm{ind}A$, thus there are no morphisms from $\mathrm{ind}\mathcal{B}_{k}$ to $\mathrm{ind}\mathcal{B}_{k+2}$ for $k=0,1, \dots, n-3$.

	$(2)$ Every object $X$ in $\mathcal{B}_{n-1}$ is of the form $X=F(W[n-1])$ for a module $W\in \mathcal{W}_{-(n-2)}$. Consider the injective envelope $\widetilde{e}:W[n-1]\rightarrow \widetilde{E}$ in $\mathcal{H}$. We have to show that $\widetilde{C}=\mathrm{Coker}_{\mathcal{H}} \widetilde{e}$ is an injective object in $\mathcal{H}$.

	Since $\mathcal{B}_{n-1}$ is closed under successors, $F(\widetilde{E})$ and $F(\widetilde{C})$ are in $\mathcal{B}_{n-1}$, hence $\widetilde{E}=E[n-1]$ and $\widetilde{C}=C[n-1]$ with $E,C\in \mathcal{W}_{-(n-2)}$ and $\widetilde{e}=e[n-1]$ for a morphism $e:W\rightarrow E$ in $\mathrm{Mod}A$. Moreover, notice that $\mathrm{Ext}^{1}_{A}(\mathcal{W}_{-(n-2)},E)\cong \mathrm{Hom}_{D^{b}(A)}(\mathcal{W}_{-(n-2)},E[1])\cong \mathrm{Ext}^{1}_{\mathcal{H}}(\mathcal{W}_{-(n-2)}[n-1],\widetilde{E})=0$ since $\widetilde{E}$ is injective in $\mathcal{H}$, hence $E\in (\mathcal{W}_{-(n-2)})^{\bot_{1}}$. Now we consider $\widetilde{e}=e[n-1]:W[n-1]\rightarrow \widetilde{E}$ in $D(A)$. Its cone is $Z:=\mathrm{Cone}(\widetilde{e})=\mathrm{Ker}(e)[n]\oplus \mathrm{Coker}(e)[n-1]$. On the other hand, we know that $\mathrm{Ker}_{\mathcal{H}}(\widetilde{e})=0$ and $\mathrm{Coker}_{\mathcal{H}}(\widetilde{e})=\widetilde{C}$ in $\mathcal{H}$. By Lemma \ref{LM}, we have

	$\widetilde{C}=Z\in \mathcal{W}[1]=\left\{X\in D(A) \ \vert \ H^{-1}(X)\in \mathcal{W}_{0},\dots,H^{-(n-1)}(X)\in \mathcal{W}_{-(n-2)},\right. \\ \left. \ \ \ \ \ \ \ \ \ \ \ \ \ \ \ \ \ \ \ \ \ \ \ \ \ \ \ \ \ \ \ \ \ \ \ \ \ \ \ \ \ \ \ \ \ \ \ \ \ \  H^{k}(X)=0 \ \forall \ k\leq-n \right\},$

	\noindent hence $\mathrm{Ker}(e)=0$ and $\mathrm{Coker}(e)=C\in \mathcal{W}_{-(n-2)}$. It follows that there is an exact sequence in $\mathrm{Mod}A$
	$$0\rightarrow W\stackrel{e}{\rightarrow} E\rightarrow C\rightarrow 0,$$
	where $E,C\in \mathcal{W}_{-(n-2)}\cap (\mathcal{W}_{-(n-2)})^{\bot_{1}}$ (recall that $(\mathcal{W}_{-(n-2)})^{\bot_{1}}$ is closed under quotients since $\mathcal{W}_{-(n-2)}\subset \mathrm{Mod}A$ consists of modules of projective dimension at most one). This shows that $\widetilde{C}$ is injective in $\mathcal{H}$. Indeed, if we take its injective envelope and consider the exact sequence
	$$0\rightarrow \widetilde{C}\rightarrow \widetilde{E^{\prime}}\rightarrow \widetilde{C^{\prime}}\rightarrow 0$$
	in $\mathcal{H}$, using the same argument as above, we see that it in fact comes from an exact sequence $0\rightarrow C\rightarrow E^{\prime}\rightarrow C^{\prime}\rightarrow 0$ in $\mathrm{Mod}A$ with $E^{\prime},C^{\prime}\in \mathcal{W}_{-(n-2)}\cap (\mathcal{W}_{-(n-2)})^{\bot_{1}}$, which must split because $C^{\prime}\in \mathcal{W}_{-(n-2)}$ and $C\in (\mathcal{W}_{-(n-2)})^{\bot_{1}}$ imply $\mathrm{Ext}^{1}_{A}(C^{\prime},C)=0$.

	$(3)$ is shown by using the dual arguments of the proof of $(2)$.
\end{proof}

\begin{theorem}\label{nsec}
	Let $T$ be an $n$-term silting complex in $D^{b}(\mathrm{mod}A)$ over a finite dimensional hereditary algebra $A$ and let  $B=\mathrm{End}_{D^{b}(\mathrm{mod}A)}T$ be the endomorphism ring of $T$. Then $\mathrm{ind}B$ has a separated $n$-section $(\mathrm{ind}\mathcal{B}_{0}, \mathrm{ind}\mathcal{B}_{1}, \dots, \mathrm{ind}\mathcal{B}_{n-1})$ with $\mathrm{ind}\mathcal{B}_{0}\subset \mathcal{L}_{B}$ and $\mathrm{ind}\mathcal{B}_{n-1}\subset \mathcal{R}_{B}$.  
\end{theorem}

\begin{proof}
	By Proposition \ref{P1}, we know that $\mathrm{ind}\mathcal{B}_{0}$ is closed under predecessors and the projective dimension is at most 1. Hence we have that $\mathrm{ind}\mathcal{B}_{0}\subset \mathcal{L}_{B}$. In the same way, since $\mathrm{ind}\mathcal{B}_{n-1}$ is closed under successors and the injective dimension is at most 1, it follows that $\mathrm{ind}\mathcal{B}_{n-1}\subset \mathcal{R}_{B}$.
\end{proof}

\begin{lemma}\label{LL}
	$(\mathrm{Gen}\mathcal{V}_{n},\mathrm{Cogen}\mathcal{W}_{-n})$ is a silting torsion pair in $\mathrm{Mod}A$ with silting module $T_{n}=B_{n}\oplus \mathrm{Coker}\lambda_{n}$.
\end{lemma}

\begin{proof}
	$\mathcal{D}_{n}=\mathrm{Gen}\mathcal{V}_{n}={^{\bot_{0}}}\mathrm{Cogen}\mathcal{W}_{-n}={^{\bot_{0}}}({\mathcal{V}_{n}}^{\bot_{0}})$ by \cite[Theorem 5.9 and 6.7]{AH}.
\end{proof}

\begin{proposition}\label{P2}
	In the $n$-section $(\mathrm{ind}\mathcal{B}_{0}, \mathrm{ind}\mathcal{B}_{1}, \dots, \mathrm{ind}\mathcal{B}_{n-1})$, the category $\mathrm{add}\mathcal{B}_{0}$ is covariantly finite and $\mathrm{add}\mathcal{B}_{n-1}$ is contravariantly finite in $\mathrm{mod}B$.
\end{proposition}

\begin{proof}
	This follows from the fact that $(\mathrm{add}(\mathcal{B}_{1}\cup\cdots\cup\mathcal{B}_{n-1}), \mathrm{add}\mathcal{B}_{0})$ and $(\mathrm{add}\mathcal{B}_{n-1}, \mathrm{add}(\mathcal{B}_{0}\cup\cdots\cup\mathcal{B}_{n-2}))$ are split torsion pairs in $\mathrm{mod}B$. 
\end{proof}

\begin{theorem}
	When $n=3$, the trisection $(\mathrm{ind}\mathcal{B}_{0}, \mathrm{ind}\mathcal{B}_{1}, \mathrm{ind}\mathcal{B}_{2})$  in Theorem~\ref{nsec} is functorially finite. 
\end{theorem}

\begin{proof}
	For $n=3$, we know that $$\mathcal{B}_{0}=F((\mathcal{V}_{0}\cap \mathrm{mod}A)[0]), \ \mathcal{B}_{1}=F( (\mathcal{V}_{1}\cap\mathcal{W}_{0}\cap \mathrm{mod}A)[1]), \ \mathcal{B}_{2}=F((\mathcal{W}_{-1}\cap \mathrm{mod}A)[2]).$$

	Note that $\mathcal{D}_{0}=\mathrm{Gen}\mathcal{V}_{0}\subset \mathcal{D}_{1}=\mathcal{V}_{1} \ \text{and} \ \mathrm{Cogen}\mathcal{W}_{-1}\subset \mathrm{Cogen}\mathcal{W}_{0}=\mathcal{W}_{0}$.
	By Lemma \ref{LL}, we have two torsion pairs in $\mathrm{Mod}A$
	$$t_{0}=(\mathrm{Gen}\mathcal{V}_{0},\mathcal{W}_{0}),t_{1}=(\mathcal{V}_{1},\mathrm{Cogen}\mathcal{W}_{-1})$$
	
	We first prove that $\mathrm{add}\mathcal{B}_{1}$ is covariantly finite in $\mathrm{mod}B$.

	Step 1: Every $M\in \mathrm{mod}A$ admits a right $(\mathcal{V}_{1}\cap\mathcal{W}_{0}\cap \mathrm{mod}A)$-approximation.

	For any $M\in \mathrm{mod}A$, take $f:W\rightarrow M$ such that $f$ is a right $(\mathcal{W}_{0}\cap \mathrm{mod}A)$-approximation of $M$, which is guaranteed by the functorially finite torsion pair $(\mathrm{gen}T_{0},\mathcal{W}_{0}\cap \mathrm{mod}A)$ in $\mathrm{mod}A$ induced by the minimal silting $A$-module $T_{0}$. Let $\mathrm{tr}(W)$ be the trace of $\mathcal{V}_{1}$ in $W$, that is, the sum of the images of all homomorphism from modules in $\mathcal{V}_{1}$ to $W$. Because $\mathcal{V}_{1}$ is closed under images and direct (hence arbitrary) sums, $\mathrm{tr}(W)$ is the largest submodule of $W$ that lies in $\mathcal{V}_{1}$. Denote the inclusion by $\iota: \mathrm{tr}(W)\rightarrow W$. On the other hand, $\mathcal{W}_{0}\cap \mathrm{mod}A$ is a torsion-free class, which is closed under submodules. Hence $\mathrm{tr}(W)$ lies in $\mathcal{V}_{1}\cap\mathcal{W}_{0}\cap \mathrm{mod}A$. It is easy to check that $h:=\iota\circ f$ is a right $(\mathcal{V}_{1}\cap\mathcal{W}_{0}\cap \mathrm{mod}A)$-approximation of $M$.

	Step 2: Let $\left\{V_{i} \ \vert \ i\in I\right\}$ be a complete irredundant set of representatives of the isomorphism classes of $\mathcal{V}_{1}\cap\mathcal{W}_{0}\cap \mathrm{mod}A$, and put $V=\bigoplus_{i\in I}V_{i}$. It is easy to see that $V \in \mathcal{V}_{1}\cap \mathcal{W}_{0}$ since the torsion class $\mathcal{V}_{1}$ is closed under direct sums and the torsion-free class $\mathcal{W}_{0}$ is closed under direct products and submodules. By Step 1, any module $M\in \mathrm{mod}A$ has an $\mathrm{add}V$-precover, which means that $\mathrm{Hom}_{A}(V,M)_{\mathrm{End}V}$ is finitely generated as right $\mathrm{End}V$-module by \cite[Lemma 3]{A}.

	Step 3: For any module $X\in \mathrm{mod}A$, there is an $\mathrm{End}_{A}V$-linear isomorphism
	${_{\mathrm{End}_{A}V}}\mathrm{Ext}^{1}_{A}(X,V)\stackrel{\phi}\cong {_{\mathrm{End}_{A}V}}D\mathrm{Hom}_{A}(V,\tau X)$.

	Obviously, $\mathrm{Hom}_{A}(V,\tau X)$ is naturally endowed with a structure of right $\mathrm{End}_{A}V$-module, by defining $(fs)(v):=f(s(v))$ for any $f\in \mathrm{Hom}_{A}(V,\tau X),\ s\in \mathrm{End}_{A}V$ and $v\in V$. Hence $D\mathrm{Hom}_{A}(V,\tau X)$ has a left $\mathrm{End}_{A}V$-module structure via the formula $(s\alpha)f:=\alpha(fs)$ for any $\alpha\in D\mathrm{Hom}_{A}(V,\tau X),\ s\in \mathrm{End}_{A}V$ and $f\in \mathrm{Hom}_{A}(V,\tau X)$.

	$\mathrm{Ext}^{1}_{A}(X,V)$ is also a left $\mathrm{End}_{A}V$-module via the map $\mathrm{End}_{A}V \times \mathrm{Ext}^{1}_{A}(X,V)\longrightarrow \mathrm{Ext}^{1}_{A}(X,V)$, $(s,[\varepsilon])\longmapsto [s\cdot \varepsilon]$, see, for example, \cite[I, $\S 5$]{ARS}.

	We have an isomorphism $\mathrm{Ext}^{1}_{A}(X,V)\cong D\mathrm{Hom}_{A}(V,\tau X)$, which is functorial in both variables as $A$-modules, see \cite{Krause}. Therefore, for $s\in \mathrm{End}_{A}V$, we have the following commutative diagram:
	$$\xymatrix{\mathrm{Ext}^{1}_{A}(X,V)\ar[r]^{\phi_{X,V}}\ar[d]_{\mathrm{Ext}^{1}_{A}(X,s)}&D\mathrm{Hom}_{A}(V,\tau X)\ar[d]^{D\mathrm{Hom}_{A}(s,\tau X)}\\
		\mathrm{Ext}^{1}_{A}(X,V)\ar[r]^{\phi_{X,V}}&D\mathrm{Hom}_{A}(V,\tau X)
	}$$
	
	For any $[\varepsilon]\in \mathrm{Ext}^{1}_{A}(X,V)$, we get $\phi_{X,V}(s\cdot [\varepsilon])=D\mathrm{Hom}_{A}(s,\tau X)(\phi_{X,V}([\varepsilon]))=\phi_{X,V}([\varepsilon])\circ \mathrm{Hom}_{A}(s,\tau X)=s\cdot \phi_{X,V}([\varepsilon])$. It follows that $\phi_{X,V}$ is $\mathrm{End}_{A}V$-linear.

	Step 4: Note that there are no morphisms from right to the left in the trisection $(\mathrm{ind}\mathcal{B}_{0}, \mathrm{ind}\mathcal{B}_{1}, \mathrm{ind}\mathcal{B}_{2})$. It suffices to show that there exists a left $\mathrm{add}\mathcal{B}_{1}$-approximation for any module in $\mathrm{ind}\mathcal{B}_{0}$.

	Consider a module $B_{0}\in \mathrm{ind}\mathcal{B}_{0}$. There exists $X\in \mathcal{V}_{0}\cap\mathrm{mod}A$ such that $B_{0}=FX$. It follows from Step 2 that $\mathrm{Hom}_{A}(V,\tau X)_{\mathrm{End}V}$ is finitely generated. Thus $ {_{\mathrm{End}_{A}V}}D\mathrm{Hom}_{A}(V,\tau X)$ is finitely generated. By Step 3, we have isomorphisms
	$${_{\mathrm{End}_{A}V}}D\mathrm{Hom}_{A}(V,\tau X)\cong {_{\mathrm{End}_{A}V}}\mathrm{Ext}^{1}_{A}(X,V)\cong {_{\mathrm{End}_{A}V}}\mathrm{Hom}_{D^{b}(\mathrm{mod}A)}(X,V[1])$$
	$$\ \ \ \ \ \ \ \ \ \ \ \ \ \ \ \ \ \ \ \ \ \ \ \ \ \cong {_{\mathrm{End}_{A}V[1]}}\mathrm{Hom}_{\mathcal{H}^{\prime}}(X,V[1])\cong {_{\mathrm{End}_{B}FY}}\mathrm{Hom}_{B}(B_{0},FY)$$
	where $Y:=V[1]=\bigoplus_{i\in I}V_{i}[1]$ and
	$\left\{F(V_{i}[1]) \ \vert \ i\in I\right\}$ is a complete irredundant set of representatives of the isomorphism classes of $\mathcal{B}_{1}$. It follows that the left $\mathrm{End}_{B}FY$-module ${_{\mathrm{End}_{B}FY}}\mathrm{Hom}_{B}(B_{0},FY)$ is finitely generated. Consequently, $B_{0}$ has a left $\mathrm{add}\mathcal{B}_{1}$-approximation by \cite[Lemma 3]{A}, as desired.

	Now we show that $\mathrm{add}\mathcal{B}_{1}$ is contravariantly finite in $\mathrm{mod}B$.

	Step $1^{\prime}$: Every $M\in \mathrm{mod}A$ admits a left $(\mathcal{V}_{1}\cap\mathcal{W}_{0}\cap \mathrm{mod}A)$-approximation.

	For any $M\in \mathrm{mod}A$, since $\mathcal{V}_{1}\cap \mathrm{mod}A$ is functorially finite, there exists $f:M\longrightarrow V$ with $V\in \mathcal{V}_{1}$ such that $f$ is a left $\mathcal{V}_{1}\cap \mathrm{mod}A$-approximation of $M$. Let $tr(V)$ be the trace of $\mathrm{Gen}\mathcal{V}_{0}$ in $V$ and  $\pi:V\longrightarrow V/tr(V)$ be the canonical epimorphism. Obviously, $V/tr(V)\in (\mathcal{V}_{1}\cap\mathcal{W}_{0}\cap \mathrm{mod}A)$ since $\mathcal{V}_{1}$ is closed under quotients. It is easy to check that $\pi f:M\longrightarrow V/tr(V)$ is a left $(\mathcal{V}_{1}\cap\mathcal{W}_{0}\cap \mathrm{mod}A)$-approximation of $M$.

	Step $2^{\prime}$: Let $V$ be the module as in the Step 2. By Step $1^{\prime}$, any module $M\in \mathrm{mod}A$ has an $\mathrm{add}V$-preenvelope, which means that ${_{\mathrm{End}V}}\mathrm{Hom}_{A}(M,V)$ is finitely generated as left $\mathrm{End}V$-module by \cite[Lemma 3]{A}.

	Step $3^{\prime}$: For any module $B_{2}\in \mathrm{ind}\mathcal{B}_{2}$, there exists $X\in (\mathcal{W}_{-1}\cap \mathrm{mod}A)$ such that $B_{2}=F(X[2])$. It follows from Step $2^{\prime}$ that ${_{\mathrm{End}V}}\mathrm{Hom}_{A}(\tau^{-1}X, V)$ is finitely generated. Thus $D\mathrm{Hom}_{A}(\tau^{-1}X, V)_{\mathrm{End}_{A}V}$ is finitely generated. Moreover, we have isomorphisms
	$$D\mathrm{Hom}_{A}(\tau^{-1}X, V)_{\mathrm{End}_{A}V}\cong \mathrm{Ext}^{1}_{A}(V,X)_{\mathrm{End}_{A}V}\cong \mathrm{Hom}_{D^{b}(\mathrm{mod}A)}(V[1],X[2])_{\mathrm{End}_{A}V[1]}$$
	$$\ \ \ \ \ \ \ \ \ \ \ \ \ \ \ \ \ \ \ \ \ \ \ \ \ \ \ \cong \mathrm{Hom}_{\mathcal{H}^{\prime}}(V[1],X[2])_{\mathrm{End}_{A}V[1]}\cong\mathrm{Hom}_{B}(FY,B_{2})_{\mathrm{End}_{B}FY}$$
	where $Y:=V[1]=\bigoplus_{i\in I}V_{i}[1]$ and
	$\left\{F(V_{i}[1]) \ \vert \ i\in I\right\}$ is a complete irredundant set of representatives of the isomorphism classes of $\mathcal{B}_{1}$. It follows that the right $\mathrm{End}_{B}FY$-module $\mathrm{Hom}_{B}(FY,B_{2})_{\mathrm{End}_{B}FY}$ is finitely generated. Consequently, $B_{2}$ has a right $\mathrm{add}\mathcal{B}_{1}$-approximation by \cite[Lemma 3]{A}, as desired.

	Therefore, $\mathrm{add}\mathcal{B}_{1}$ is a functorially finite subcategory of $\mathrm{mod}B$. The assertion follows from Proposition \ref{P2} and \cite[Sec. 2.6 Theorem]{AACPT}.
\end{proof}

Finally, we note that every functorially finite $n$-section over a hereditary algebra, under mild conditions, is associated to an $n$-term silting complex.

\begin{proposition}
	If $A$ is a hereditary algebra with a functorially finite $n$-section $(\mathcal{A}_{1}, \mathcal{A}_{2}, \dots, \mathcal{A}_{n})$ in $\mathrm{mod}A$ and $\mathrm{Ext}_{A}^{1}(\mathcal{A}_{i}, \mathcal{A}_{j})=0$ for $i-j\geq 2$, then there exists a chain of finite dimensional homological ring epimorphisms $\lambda_{i}:A\rightarrow B_{i}$ such that $0_{A}\leq\lambda_{1}\leq \cdots \leq\lambda_{n-1}\leq id_{A}$ and $\mathrm{gen}B_{1}=\mathrm{add}\mathcal{A}_{n},\mathrm{gen}B_{2}=\mathrm{add}(\mathcal{A}_{n-1}\cup\mathcal{A}_{n}),\ldots,\mathrm{gen}B_{n-1}=\mathrm{add}(\mathcal{A}_{2}\cup\cdots\cup\mathcal{A}_{n})$.
\end{proposition}

\begin{proof}
	Since every $\mathrm{add}\mathcal{A}_{i}$ is functorially finite, there are $n-1$ split torsion pairs with functorially finite torsion classes
	$$(\mathrm{add}\mathcal{A}_{n},\mathrm{add}(\mathcal{A}_{1}\cup\cdots\cup \mathcal{A}_{n-1})), (\mathrm{add}(\mathcal{A}_{n-1}\cup\mathcal{A}_{n}),\mathrm{add}(\mathcal{A}_{1}\cup\cdots\cup \mathcal{A}_{n-2})),\ldots,$$
	$$(\mathrm{add}(\mathcal{A}_{3}\cup\cdots\cup\mathcal{A}_{n}),\mathrm{add}(\mathcal{A}_{1}\cup \mathcal{A}_{2})),(\mathrm{add}(\mathcal{A}_{2}\cup\cdots\cup\mathcal{A}_{n}),\mathrm{add}\mathcal{A}_{1})$$
	in $\mathrm{mod}A$. By \cite[Theorem 2.7]{AIR}, there exist $n-1$ finite dimensional silting modules $T_{i}$ such that 
	$$\mathrm{gen}T_{1}=\mathrm{add}\mathcal{A}_{n},\mathrm{gen}T_{2}=\mathrm{add}(\mathcal{A}_{n-1}\cup\mathcal{A}_{n}),\ldots,\mathrm{gen}T_{n-2}=\mathrm{add}(\mathcal{A}_{3}\cup\cdots\cup\mathcal{A}_{n}),$$
	$$\mathrm{gen}T_{n-1}=\mathrm{add}(\mathcal{A}_{2}\cup\cdots\cup\mathcal{A}_{n}).$$
	It follows from \cite[(4.4) Lemma]{CB} that there are torsion pairs in $\mathrm{Mod}A$
	$$(\mathrm{Gen}T_{1}=\varinjlim \mathcal{A}_{n},\varinjlim (\mathcal{A}_{1}\cup\cdots\cup \mathcal{A}_{n-1})),\ldots,(\mathrm{Gen}T_{n-1}=\varinjlim(\mathcal{A}_{2}\cup\cdots\cup\mathcal{A}_{n}),\varinjlim\mathcal{A}_{1}).$$
	By \cite[Corollary 5.12]{AMV2}, every $T_{i}$ corresponds to a homological ring epimorphism $\lambda_{i}:A\longrightarrow B_{i}$ for $1\leq i\leq n-1$ and then we have that  $\mathrm{Mod}B_{i}=\alpha(\mathrm{Gen}T_{i}):=\left\{X\in \mathrm{Gen}T_{i} \ \vert \ \forall (g:Y\rightarrow X)\in \mathrm{Gen}T_{i}, \mathrm{Ker}g\in \mathrm{Gen}T_{i}\right\}$ by \cite[Proposition 5.6]{AMV2}.

	We need to show that $\lambda_{i}\leq\lambda_{i+1}$, which is equivalent to $\alpha(\mathrm{Gen}T_{i})\subseteq \alpha(\mathrm{Gen}T_{i+1})$. So it suffices to show that $\mathrm{Gen}T_{i}\subseteq \alpha(\mathrm{Gen}T_{i+1})$. For any $X\in \mathrm{Gen}T_{i}$, consider $g:Y\rightarrow X$ in $\mathrm{Gen}T_{i+1}$. We want to prove $K:=\mathrm{Ker}g\in \mathrm{Gen}T_{i+1}$. In fact, by \cite[Proposition 2.14]{SF},
	$$\mathrm{Gen}T_{i}=\varinjlim (\mathcal{A}_{n-i+1}\cup\cdots\cup\mathcal{A}_{n})={^{\bot_{0}}}(\mathcal{A}_{1}\cup\cdots\cup \mathcal{A}_{n-i}),$$
	$$\mathrm{Gen}T_{i+1}=\varinjlim (\mathcal{A}_{n-i}\cup\cdots\cup\mathcal{A}_{n})={^{\bot_{0}}}(\mathcal{A}_{1}\cup\cdots\cup \mathcal{A}_{n-i-1}).$$
	For any $A^{\prime}\in \mathcal{A}_{1}\cup\cdots\cup \mathcal{A}_{n-i-1}$, apply $\mathrm{Hom}_{A}(-,A^{\prime})$ to the short exact sequences
	$$0\longrightarrow K\longrightarrow Y\longrightarrow \mathrm{Im}g\longrightarrow 0,$$
	$$0\longrightarrow \mathrm{Im}g\longrightarrow X\longrightarrow \mathrm{Coker}g\longrightarrow 0,$$
	respectively. We obtain two long exact sequences 
	$$\cdots\longrightarrow \mathrm{Hom}_{A}(Y,A^{\prime})\longrightarrow \mathrm{Hom}_{A}(K,A^{\prime})\longrightarrow \mathrm{Ext}^{1}_{A}(\mathrm{Im}g,A^{\prime})\longrightarrow\cdots,$$
	$$\cdots\longrightarrow\mathrm{Ext}^{1}_{A}(X,A^{\prime})\longrightarrow \mathrm{Ext}^{1}_{A}(\mathrm{Im}g,A^{\prime})\longrightarrow \mathrm{Ext}^{2}_{A}(\mathrm{Coker}g,A^{\prime})\longrightarrow\cdots.$$
	Obviously, $\mathrm{Hom}_{A}(Y,A^{\prime})=0$ since $Y\in \mathrm{Gen}T_{i+1}={^{\bot_{0}}}(\mathcal{A}_{1}\cup\cdots\cup \mathcal{A}_{n-i-1}),\ A^{\prime}\in \mathcal{A}_{1}\cup\cdots\cup \mathcal{A}_{n-i-1}$, and $\mathrm{Ext}^{2}_{A}(\mathrm{Coker}g,A^{\prime})=0$ since $A$ is hereditary. Moreover, since every finite dimensional module is pure-injective and $\mathrm{Ext}_{A}^{1}(\mathcal{A}_{i}, \mathcal{A}_{j})=0$ for $i-j\geq 2$, we have the isomorphism 
	$$\mathrm{Ext}^{1}_{A}(\varinjlim (\mathcal{A}_{n-i+1}\cup\cdots\cup\mathcal{A}_{n}),\mathcal{A}_{1}\cup\cdots\cup \mathcal{A}_{n-i-1})\cong \varprojlim \mathrm{Ext}^{1}_{A}(\mathcal{A}_{n-i+1}\cup\cdots\cup\mathcal{A}_{n},\mathcal{A}_{1}\cup\cdots\cup \mathcal{A}_{n-i-1})=0.$$
	Hence $\mathrm{Ext}^{1}_{A}(X,A^{\prime})=0$ and then $\mathrm{Ext}^{1}_{A}(\mathrm{Im}g,A^{\prime})=0$. Consequently $\mathrm{Hom}_{A}(K,A^{\prime})=0$, i.e. $K\in \mathrm{Gen}T_{i+1}$, as desired.
\end{proof}

\subsection{Examples}
We close this note with some examples. Let us first recall some generalizations of the notion of a tilted algebra appearing in the literature.  

\begin{definition}
	
	An algebra $A$ is called 
	\\
	$\bullet$ {\it quasi-tilted} \cite{HRS} if it satisfies
	\\
	$(i)$ $\mathrm{gl.dim}A\leq 2$; and \\
	$(ii)$ $\mathrm{ind}A=\mathcal{L}_{A}\cup \mathcal{R}_{A}$, or  equivalently, $\mathrm{pd}X\leq 1$ or  $\mathrm{id}X\leq 1$ for each  $X\in\mathrm{ind}A$, cf.~\cite{CL};  
	\\
	$\bullet$ {\it shod} \cite{CL} if it satisfies  condition $(ii)$ above;
	\\
	$\bullet$ {\it weakly shod} \cite{CL2} provided \\$(iii)$ $\mathcal{L}_{A}\cup \mathcal{R}_{A}$ is cofinite in $\mathrm{ind}A$, i.e. $\mathrm{ind}A\setminus (\mathcal{L}_{A}\cup \mathcal{R}_{A})$ has finitely many objects; and\\$(iv)$ no nonsemiregular component of the Auslander-Reiten quiver has oriented cycles;
\\
	$\bullet$ {\it laura} \cite{AC} if it satisfies  condition $(iii)$ above. 
\end{definition}

We don't know whether all 3-silted algebras are  laura. The following example exhibits a 3-silted algebra which is weakly shod, but not shod.

\begin{example}
	Consider the hereditary algebra $A=kQ$ with the quiver
	$$\xymatrix@=0,4cm{ &1\ar[rd]& & & \\
	Q:& &3\ar[r]&4\ar[r]&5 \\
	  &2\ar[ru]& & & \\}$$
	Let $P_{i}$ (resp., $I_{i}$) denotes the projective (resp., injective) $A$-module corresponding to vertex $i$. Consider the universal localisation of $A$ at $I_{3}\oplus I_{2}\oplus I_{1}$. By \cite[Theorem 6.1]{K}, it is a homological ring epimorphism, denoted by $\lambda_{0}:A\longrightarrow B_{0}:=A_{\left\{I_{3}\oplus I_{2}\oplus I_{1}\right\}}$, which corresponds to the bireflective subcategory $\mathcal{X}_{0}=(I_{3}\oplus I_{2}\oplus I_{1})^{\bot_{0,1}}=\mathrm{Add}(P_{5}\oplus I_{5}\oplus I_{4})$. Furthermore, $B_{0}={I_{5}}^{4}\oplus P_{5}$ and the minimal silting module corresponding to the homological ring epimorphism is equivalent to $T_{0}=I_{5}\oplus P_{5}\oplus I_{3}\oplus I_{2}\oplus I_{1}$. Similarly, consider another universal localisation of $A$ at $\begin{smallmatrix}
		1 \\ 3
	\end{smallmatrix}\oplus I_{1}$. Then the homological ring epimorphism $\lambda_{1}:A\longrightarrow B_{1}:=A_{\left\{\begin{smallmatrix}
		1 \\ 3
	\end{smallmatrix}\oplus I_{1}\right\}}$ corresponds to the bireflective subcategory $\mathcal{X}_{1}=(\begin{smallmatrix}
	1 \\ 3
    \end{smallmatrix}\oplus I_{1})^{\bot_{0,1}}=\mathrm{Add}(P_{5}\oplus P_{1}\oplus \begin{smallmatrix}
    1 \\ 3 \\ 4
    \end{smallmatrix}\oplus I_{5}\oplus I_{4}\oplus I_{2})$. It is easy to check that $B_{1}={P_{1}}^{3}\oplus I_{5}\oplus P_{5}$ and the minimal silting module is equivalent to $T_{1}=P_{1}\oplus I_{5}\oplus P_{5}\oplus I_{1}\oplus \begin{smallmatrix}
    1 \\ 3
    \end{smallmatrix}$. Clearly, the partial order $\mathcal{X}_{0}\subseteq \mathcal{X}_{1}$ corresponds, under the bijection in Theorem \ref{Th1}, to the partial order $\lambda_{0}\leq \lambda_{1}$ and there exists a ring epimorphism $\mu:B_{1}\longrightarrow B_{0}$ such that $\lambda_{0}=\mu \circ\lambda_{1}$. It is easy to verify that $\mathrm{Ker}\mu=0,\mathrm{Coker}\mu={I_{2}}^{3},\mathrm{Ker}\lambda_{1}=0,\mathrm{Coker}\lambda_{1}={I_{1}}^{2}\oplus \begin{smallmatrix}
    	1 \\ 3
    \end{smallmatrix}$. Then, by Proposition \ref{Pro}, we get the $3$-term silting complex $$T=B_{0}\oplus \mathrm{Ker}\mu[1]\oplus \mathrm{Coker}\mu\oplus \mathrm{Ker}\lambda_{1}[2]\oplus \mathrm{Coker}\lambda_{1}[1]\simeq P_{5}\oplus I_{5}\oplus I_{2}\oplus \begin{smallmatrix}
    1 \\ 3
    \end{smallmatrix}[1]\oplus I_{1}[1].$$
    It is easy to see that the endomorphism ring $B=\mathrm{End}_{D^{b}(\mathrm{mod}A)}T$ of $T$ is isomorphic to the algebra $kQ^{\prime}/I^{\prime}$ with $Q^{\prime}:\bullet\stackrel{\alpha}\longrightarrow\bullet\stackrel{\beta}\longrightarrow\bullet\stackrel{\gamma}\longrightarrow\bullet\stackrel{\eta}\longrightarrow \bullet$ and $I^{\prime}=\langle\alpha\beta,\beta\gamma,\gamma\eta\rangle$. Note that it is weakly shod, but it is not a shod algebra since the global dimension of $B$ is $4$.
\end{example}

We also include an example of a (shod) 3-silted algebra arising from universal localization at regular modules. 

\begin{example}
	Let $A=k \widetilde{D_{6}}$ be the path algebra of 
	$$\xymatrix@=0,4cm{ &1& & & &6\ar[ld]\\
		\widetilde{D_{6}}:& &3\ar[lu]\ar[ld]&4\ar[l]&5\ar[l] & \\
		&2& & & &7\ar[lu]\\}$$
	It is well known that the Auslander-Reiten quiver $\Gamma(\mathrm{mod}A)$ of $A$ is the union of the postprojective component $\mathcal{P}(A)$, the preinjective component $\mathcal{Q}(A)$ and the union $\mathcal{R}(A)$ of regular components.
	Let $P_{i}$ denotes the projective $A$-module corresponding to vertex $i$. Consider the stable tube of rank $4$ in $\mathcal{R}(A)$ containing the modules $F_{1}^{(1)}, \ldots , F_{4}^{(1)}$ and, for each $s\in \left\{1,\ldots,4\right\}$, there is an isomorphism $\tau_{A}F_{s+1}^{(1)}\cong F_{s}^{(1)}$, where we set $F_{5}^{(1)}=F_{1}^{(1)}$.
	
	$$\xymatrix@=0,4cm{ &0& & & &0\ar[ld]\\
		F_{1}^{(1)}:& &K\ar[lu]\ar[ld]&0\ar[l]&0\ar[l] & \\
		&0& & & &0\ar[lu]\\} \ \ \ \ \ \ \ \ \ \ 
		\xymatrix@=0,4cm{ &0& & & &0\ar[ld]\\
			F_{2}^{(1)}:& &0\ar[lu]\ar[ld]&K\ar[l]&0\ar[l] & \\
			&0& & & &0\ar[lu]\\}$$
			
	$$\xymatrix@=0,4cm{ &0& & & &0\ar[ld]\\
		F_{3}^{(1)}:& &0\ar[lu]\ar[ld]&0\ar[l]&K\ar[l] & \\
		&0& & & &0\ar[lu]\\} \ \ \ \ \ \ \ 
	\xymatrix@=0,4cm{ &K& & & &K\ar[ld]_{1}\\
		F_{4}^{(1)}:& &K\ar[lu]_{1}\ar[ld]_{1}&K\ar[l]_{1}&K\ar[l]_{1} & \\
		&K& & & &K\ar[lu]_{1}\\}$$
	Consider the universal localisation of $A$ at $F_{2}^{(1)}\oplus F_{3}^{(1)}$ and denote the homological ring epimorphism by $\lambda_{0}:A\longrightarrow B_{0}:=A_{F_{2}^{(1)}\oplus F_{3}^{(1)}}$, which corresponds to the bireflective subcategory $\mathcal{X}_{0}=(F_{2}^{(1)}\oplus F_{3}^{(1)})^{\bot_{0,1}}$. It is easy to see that $B_{0}=P_{1}\oplus P_{2}\oplus{P_{5}}^{3}\oplus P_{6}\oplus P_{7}$ and the minimal silting module corresponding to the homological ring epimorphism is equivalent to $T_{0}=P_{1}\oplus P_{2}\oplus P_{5}\oplus P_{6}\oplus P_{7}\oplus F_{2}^{(1)}\oplus F_{3}^{(1)}$. Similarly, the homological ring epimorphism $\lambda_{1}:A\longrightarrow B_{1}$ corresponds to universal localisation of $A$ at $F_{2}^{(1)}$, and then the bireflective subcategory is $\mathcal{X}_{0}=(F_{2}^{(1)})^{\bot_{0,1}}$. By easy calculation, we have $B_{1}=P_{1}\oplus P_{2}\oplus{P_{4}}^{2}\oplus P_{5}\oplus P_{6}\oplus P_{7}$ and the minimal silting module corresponding to $\lambda_{1}$ is equivalent to $T_{1}=P_{1}\oplus P_{2}\oplus P_{4}\oplus P_{5}\oplus P_{6}\oplus P_{7}\oplus F_{2}^{(1)}$.  Clearly, the partial order $\mathcal{X}_{0}\subseteq \mathcal{X}_{1}$ corresponds, under the bijection in Theorem \ref{Th1}, to the partial order $\lambda_{0}\leq \lambda_{1}$ and there exists a ring epimorphism $\mu:B_{1}\longrightarrow B_{0}$ such that $\lambda_{0}=\mu \circ\lambda_{1}$. It is easy to verify that $\mathrm{Ker}\mu=0,\mathrm{Coker}\mu=(F_{3}^{(1)})^{2},\mathrm{Ker}\lambda_{1}=0,\mathrm{Coker}\lambda_{1}=F_{2}^{(1)}$. Therefore, by Proposition \ref{Pro}, we get the $3$-term silting complex $$T=B_{0}\oplus \mathrm{Ker}\mu[1]\oplus \mathrm{Coker}\mu\oplus \mathrm{Ker}\lambda_{1}[2]\oplus \mathrm{Coker}\lambda_{1}[1]\simeq P_{1}\oplus P_{2}\oplus{P_{5}}\oplus P_{6}\oplus P_{7}\oplus F_{3}^{(1)}\oplus F_{2}^{(1)}[1].$$
	It is easy to check that the endomorphism ring $B=\mathrm{End}_{D^{b}(\mathrm{mod}A)}T$ of $T$ is isomorphic to the algebra $kQ^{\prime}/I^{\prime}$ with
    $$\xymatrix@R=0.6cm@C=0.8cm{ &\bullet\ar[rd]^{\alpha_{1}}& &\bullet& \\
    	Q^{\prime}:& &\bullet\ar[ru]^{\beta_{1}}\ar[rd]_{\beta_{2}}\ar[r]^{\beta}&\bullet\ar[r]^{\gamma}&\bullet \\
    	&\bullet\ar[ru]_{\alpha_{2}}& &\bullet& \\}$$   
    and $I^{\prime}=\langle\alpha_{1}\beta,\alpha_{2}\beta,\beta\gamma\rangle $. It is a shod algebra.
	
\end{example}

	




\vskip 10pt

\noindent{\bf{Acknowledgments}} 

\vskip 10pt

This work was started during a ICTP-INdAM Research in Pairs Programme with the visit of the second and last named authors at the University of Verona in 2019. 
The first and last named authors would like to thank the  Network on Silting Theory funded by the Deutsche Forschungsgemeinschaft and the University of Stuttgart for hospitality during a research visit in 2022 where  part of this work was carried out.
The first named author also acknowledges support from the project  \textit{SQUARE: Structures for Quivers, Algebras and Representations}, PRIN~2022S97PMY,  funded by the  Italian Ministry of University and Research. The third named author is supported by China Scholarship Council (Grant No. 202306860057). The fourth named author is supported by the project PICT 2021 01154 from ANPCyT, Argentina.



\bibliographystyle{abbrv}
\bibliography{ALLT}

\end{document}